%% file: RACGs_virtcycl_v2_181127.tex
\documentclass[11pt]{amsart}

\usepackage[margin=3cm]{geometry}
\title[Subgroup growth of certain RACGs]{Subgroup growth of virtually cyclic right-angled Coxeter groups and their free products}
\author{Hyungryul Baik, Bram Petri and Jean Raimbault}
  \address{Department of Mathematical Sciences, KAIST, 291 Daehak-ro Yuseong-gu, Daejeon, 34141, South Korea }
  \email{hrbaik@kaist.ac.kr}
  \address{Mathematical Institute, University of Bonn, Endenicher Allee 60, Bonn, Germany}
  \email{bpetri@math.uni-bonn.de}
  \address{Institut de Math\'ematiques de Toulouse ; UMR5219 \\ Universit\'e de Toulouse ; CNRS \\ UPS IMT, F-31062 Toulouse Cedex 9, France}
  \email{Jean.Raimbault@math.univ-toulouse.fr}
\date{\today}

\thanks{
H.~B. was partially supported by Samsung Science \& Technology Foundation grant No. SSTF-BA1702-01.
B.~P. gratefully acknowledges support from the ERC Advanced Grant ``Moduli''. 
J.~R. was supported by the grant ANR-16-CE40-0022-01 - AGIRA. }

\usepackage{amsmath,amsfonts,amssymb,amsthm,graphicx,overpic,hyperref,mathabx,color}
\usepackage{float,enumitem}
\usepackage[toc,page]{appendix}

\newtheorem{thm}{Theorem}[section]
\newtheorem{prp}[thm]{Proposition}
\newtheorem{cor}[thm]{Corollary}
\newtheorem{lem}[thm]{Lemma}

\newtheorem{innerthmrep}{Theorem}
\newenvironment{thmrep}[1]
  {\renewcommand\theinnerthmrep{#1}\innerthmrep}
  {\endinnerthmrep}

\newtheorem{innercorrep}{Corollary}
\newenvironment{correp}[1]
  {\renewcommand\theinnercorrep{#1}\innercorrep}
  {\endinnercorrep}
  
\newtheorem{innerlemrep}{Lemma}
\newenvironment{lemrep}[1]
  {\renewcommand\theinnerlemrep{#1}\innerlemrep}
  {\endinnerlemrep}

\newtheorem{innerprprep}{Proposition}
\newenvironment{prprep}[1]
  {\renewcommand\theinnerprprep{#1}\innerprprep}
  {\endinnerprprep} 

\theoremstyle{definition}
\newtheorem{dff}[thm]{Definition}

\newcommand{\nc}{\newcommand}
\nc{\dmo}{\DeclareMathOperator}

\usepackage[margin=3cm]{geometry}
\nc{\abs}[1]{\left| #1 \right|}
\nc{\bigO}[1]{O\left(#1\right)}
\nc{\card}[1]{\left|#1\right|}
\nc{\ceil}[1]{\left\lceil #1 \right\rceil}
\nc{\CC}{\mathbb{C}}
\nc{\floor}[1]{\left\lfloor #1 \right\rfloor}
\nc{\ZZ}{\mathbb{Z}}
\nc{\len}[1]{\left| #1 \right|}
\nc{\littleo}[1]{o\left(#1\right)}
\dmo{\Mat}{Mat}
\nc{\NN}{\mathbb{N}}
\nc{\norm}[1]{\left|\left| #1 \right|\right|}
\nc{\QQ}{\mathbb{Q}}
\nc{\RR}{\mathbb{R}}
\nc{\st}[2]{\left\{ #1 ;\; #2\right\}}
\dmo{\supp}{supp}
\nc{\tr}[1]{\mathrm{tr}\left(#1\right)}

\dmo{\area}{area}
\dmo{\conv}{conv}
\dmo{\diam}{diam}
\dmo{\DD}{\mathbb{D}}
\dmo{\dist}{\mathrm{d}}
\nc{\HH}{\mathbb{H}}
\dmo{\MCG}{MCG}
\dmo{\MPL}{MPL}
\dmo{\Mod}{\mathcal{M}}
\dmo{\PL}{PL}
\nc{\Sphere}{\mathbb{S}}
\dmo{\sys}{sys}
\dmo{\Teich}{\mathcal{T}}
\nc{\Torus}{\mathbb{T}}
\dmo{\vol}{vol}
\dmo{\WP}{WP}

\dmo{\convTV}{\;\stackrel{\mathrm{TV}}{\longrightarrow}\;}
\nc{\ExV}[2]{\mathbb{E}_{#1}\left[#2\right]}
\dmo{\EE}{\mathbb{E}}
\nc{\Pro}[2]{\mathbb{P}_{#1}\left[#2\right]}
\dmo{\PP}{\mathbb{P}}
\nc{\distTV}[2]{\mathrm{d}_{\rm TV}\left(#1,#2\right)}
\dmo{\UU}{\mathbb{U}}
\nc{\Var}[2]{\mathbb{V}\mathrm{ar}_{#1}\left[#2\right]}

\dmo{\alt}{\mathfrak{A}}
\dmo{\Aut}{Aut}
\dmo{\Fix}{Fix}
\dmo{\Hom}{Hom}
\dmo{\PSL}{PSL}
\dmo{\Rep}{Rep}
\dmo{\sym}{\mathfrak{S}}
\dmo{\inv}{\mathcal{I}}
\dmo{\orb}{\mathcal{O}}
\dmo{\stab}{Stab}
\nc{\cox}{\Gamma^\mathrm{Cox}}
\nc{\art}{\Gamma^\mathrm{Art}}

\begin{document}

\begin{abstract}
  We determine the asymptotic number of index $n$ subgroups in virtually cyclic Coxeter groups and their free products as $n\to\infty$.
\end{abstract}

\maketitle

\section{Introduction}

Given a finitely generated group $\Gamma$ and a natural number $n$, the number $s_n(\Gamma)$ of index $n$ subgroups of $\Gamma$ is finite. This leads to the question how, given a group, the number $s_n(\Gamma)$ behaves as a function of $n$ and which geometric information about the group is encoded in it. It is quite rare that an explicit expression for $s_n(\Gamma)$ can be written down and even if so, the expression one obtains might still be so complicated that it's hard to extract any information out of it. As such, one usually considers the asymptotic behavior of $s_n(\Gamma)$ for large $n$. 

In our previous paper \cite{BPR}, we considered the factorial growth rate of $s_n(\Gamma)$ for right-angled Artin and Coxeter groups. That is, we studied limits of the form 
\[ \lim_{n\to\infty} \frac{\log (s_n(\Gamma))}{n\log (n)}.\]
In the case of right-angled Artin groups we were able to determine this limit explicitly and in the case of right-angled Coxeter groups, we determined it for a large class of groups. Moreover, we conjectured an explicit formula for this limit. Our methods were mainly based on counting arguments.

\subsection{New results}
In this paper, we further study the case of right-angled Coxeter groups. We consider a very specific sequence of such groups: virtually cyclic Coxeter groups and their free products. The upshot of this is that we can access much finer asymptotics than we can in the general case.

Recall that the right-angled Coxeter group associated to a graph $\mathcal{G}$ with vertex set $V$ and edge set $E$ is given by
\[\cox(\mathcal{G}) = \langle \sigma_v,\; v\in V|\: \sigma_v^2=e \; \forall v\in V,\; [\sigma_v,\sigma_w]=e\; \forall \{v,w\}\in E \rangle.\]

Since M\"uller's results \cite{Mul2} already cover all finite groups, we focus on the infinite case. It turns out that it's not hard to classify infinite virtually cyclic right-angled Coxeter groups. Indeed they are exactly those groups whose defining graph is a suspension over a complete graph $\mathcal{K}_r$ on $r\in\NN$ vertices (see Lemma \ref{lem_class}). Let us denote these graphs by $\mathcal{A}_r$.

Our first result is that for a virtually cyclic Coxeter group, an explicit formula for its number of subgoups can be written down. In this formula and throughout the paper (whenever no confusion arises from it), $s_{2^j}$ will denote the number index $2^j$ subgroups of $(\ZZ/2\ZZ)^r$ and is given by
\[s_{2^j} = s_{2^j}((\ZZ/2\ZZ)^r) = \frac{\prod_{l=0}^{j-1} (2^r-2^l)}{\prod_{l=0}^{j-1} (2^j-2^l)}. \]
for all $j=0,\ldots, r$.
\begin{correp}{\ref{cor_subgroupcount}} 
Let $r \in \NN$ and let
\[\Gamma = \cox(\mathcal{A}_r).\]
Then
\[s_n(\Gamma) = n\; \left(1+ \sum_{\substack{ 0< j \leq r \\ s.t. \;  2^j | n}} 2^j\; s_{2^j} \right) + \sum_{\substack{ 0\leq  j \leq r \\ s.t. \;  2^{j+1} | n}} 2^j\; s_{2^j},\]
for all $n\in \NN$.
\end{correp}

One thing to note in the formula above is that the number of subgroups is not a monotone function in $n$: if $n$ is divisible by a large power of $2$ then there is a jump. In particular, no smooth asymptote is to be expected for this sequence. This is very different from the situation for non-trivial free products of virtually cyclic right-angled Coxeter groups\footnote{Recall that for functions $f,g:\NN\to\RR$ the notation $f(n)\sim g(n)$ as $n\to \infty$ indicates that $f(n)/g(n)\to 1$ as $n\to\infty$.}:

\begin{thmrep}{\ref{thm_subgrps}}
Let $m\in \NN_{\geq 2}$, $r_1,\ldots,r_m \in \NN$ and
\[\Gamma = \Asterisk_{l=1}^m \cox(\mathcal{A}_{r_l}) .\]
Then there exist explicit constants $A_\Gamma, B_\Gamma >0$ and 
\[C_\Gamma \geq 0\]
with equality if and only if $r_l \in \{0,1\}$ for $l=1,\ldots,m$ (see Definition \ref{def_csts}) so that
\[ s_n(\Gamma) \sim A_\Gamma \;n^{1+C_\Gamma}\; \exp(B_\Gamma\;\sqrt{n}) \; n!^{m-1}.\]
as $n\to\infty$.
\end{thmrep}
For the sake of simplicity, we have not included finite factors in the free product above. But, using M\"uller's results \cite{Mul2}, the theorem above can easily be extended to also allow free products with finite factors.

Because the definitions of the constants $A_\Gamma$, $B_\Gamma$ and $C_\Gamma$  (especially the former) are rather lengthy, we will postpone them to Section \ref{sec_permreps}. We do however note that they behave nicely with respect to free products. That is
\[A_\Gamma = \prod_{l=1}^m A_{r_l},\quad B_\Gamma = \sum_{l=1}^m B_{r_l} \quad \text{and} \quad C_\Gamma = \sum_{l=1}^m C_{r_l}.\]
where $A_r = A_{\cox(\mathcal{A}_{r_l})}$, $B_r = B_{\cox(\mathcal{A}_{r_l})}$ and $C_r = C_{\cox(\mathcal{A}_{r_l})}$. The values for some low complexity cases are

\input{table_values}

We derive the asymptote above from the number of permutation representations of $\Gamma$. Define
\[h_n(\Gamma) = \card{\Hom(\Gamma,\sym_n)},\]
where $\sym_n$ denotes the symmetric group on $n$ letters. We have:

\begin{thmrep}{\ref{thm_permrep}}
Let $r_1,\ldots,r_m \in \NN$ and
\[\Gamma = \Asterisk_{l=1}^m \cox(\mathcal{A}_{r_l}) .\]
Then 
\[ h_n(\Gamma) \sim A_\Gamma \;n^{C_\Gamma}\; \exp(B_\Gamma\;\sqrt{n}) \; n!^m.\]
as $n\to\infty$.
\end{thmrep}
Note that this asymptote \emph{does} hold in the case where $m=1$. Moreover, in the case $m=1$ and $r=0$, it recovers the asynmptote implied by the classical result of Chowla, Herstein and Moore \cite{CHM}.

The results above are all based on the fact that the exponential generating function for the sequence $(h_n(\Gamma))_n$ converges. In fact, it follwos from \cite[Proposition 2.1]{BPR} that the exponential generating function for this sequence converges if and only if $\Gamma$ is virtually abelian. Let us write $\Gamma = \cox(\mathcal{A}_r)$ and
\[G_r(x) = \sum_{n=0}^\infty \frac{h_n(\Gamma)}{n!}\;x^n.\]

We have:
\begin{thmrep}{\ref{thm_genfn}} Let $r\in \NN$. Then
\[G_r(x) = \prod_{j=0}^r \left( \left(1- x^{2^{j+1}}\right)^{-s_{2^j}/2} \exp\left(-2^j\; s_{2^j}+\frac{2^j\; s_{2^j}}{1-x^{2^j}}\right)\right)\]
where $s_{2^j}=s_{2^j}((\ZZ/2\ZZ)^r)$.
\end{thmrep}

Besides asymptotic information on the sequence $(h_n(\cox(\mathcal{A}_r)))_n$, Theorem \ref{thm_genfn} also allows us to derive the following recurrence for this sequence:

\begin{correp}{\ref{cor_recurrence}}Let $r\in\NN$ and write
\[h_{n,r} = h_n(\cox(\mathcal{A}_r)).\]
Then
\begin{eqnarray*}
 \frac{h_{n+1,r}}{n!} & = & - \sum_{\substack{\varepsilon \in I_r \setminus \{0\} \\ \abs{\varepsilon} \leq n}} \frac{h_{n-\abs{\varepsilon}+1,r}}{(n-\abs{\varepsilon})!}\; \prod_{j=0}^r (-1)^{\varepsilon_{j,1}+\varepsilon_{j,2}} 
 \\
 & & + \sum_{j=0}^r \sum_{\substack{\varepsilon \in I_r \\ \varepsilon_{j,1}=\varepsilon_{j,2} \\ =\varepsilon_{j,3}=0}} \frac{h_{n+1-2^j-\abs{\varepsilon},r}}{(n+1-2^j-\abs{\varepsilon})!} \cdot s_{2^j}\cdot 2^{2j} \; \prod_{k\neq j} (-1)^{\varepsilon_{k,1}+\varepsilon_{k,2}} 
 \\
  & & + \sum_{j=0}^r \sum_{\substack{\varepsilon \in I_r \\ \varepsilon_{j,1}=\varepsilon_{j,2}\\ =\varepsilon_{j,3}=0}}  \frac{h_{n+1-2^{j+1}-\abs{\varepsilon},r}}{(n+1-2^{j+1}-\abs{\varepsilon})!} \cdot s_{2^j}\cdot (2^j+2^{2j})\;  \prod_{k\neq j} (-1)^{\varepsilon_{k,1}+\varepsilon_{k,2}}
   \\
  & & - \sum_{j=0}^r  \sum_{\substack{\varepsilon \in I_r \\ \varepsilon_{j,1}=\varepsilon_{j,2}\\ =\varepsilon_{j,3}=0}}  \frac{h_{n+1-3\cdot 2^j-\abs{\varepsilon},r}}{(n+1-3\cdot 2^j-\abs{\varepsilon})!} \cdot s_{2^j}\cdot 2^j \; \prod_{k\neq j} (-1)^{\varepsilon_{k,1}+\varepsilon_{k,2}}
\end{eqnarray*}
for all $n\geq 0$, with the initial conditions
\[h_{0,r} = 1\;\; \text{and} \;\;h_{m,r} = 0 \;\text{for all}\; m<0,\]
where $I_r = \left(\{0,1\}^{r+1}\right)^3$, $\abs{\varepsilon} = \sum_{j=0}^r (\varepsilon_{j,1}+\varepsilon_{j,2}+\varepsilon_{j,3})\;2^j$ for all $\varepsilon\in I_r$ and $s_{2^j} = s_{2^j}((\ZZ/2\ZZ)^r)$.
\end{correp}

\subsection{Idea of the proof}

First of all, we derive a closed formula for the number of permutation representations $h_n(\Gamma)$ of an infinite virtually cyclic right-angled Coxeter group (Proposition \ref{prp_closedform}). Even though this expression is reasonably explicit, it seems hard to use it directly to extract information on the asymptotic behavior of the sequence. 

Instead, we use it to derive the exponential generating function $G_r(x)$ for the sequence. We do this by deriving recurrences for factors that appear in the expression, which then lead to an ordinary differential equation for the generating function $G_r(x)$.

Once we have determined $G_r(x)$, we use estimates on a contour integral to estimate its coefficients. It turns out that $G_r(x)$ is what is called H-admissible, which means that classical results due to Hayman \cite{Hay} allow us to determine the asymptotic behavior of its coefficients. Concretely, in Theorem \ref{thm_asymp} we prove an asymptote for a class of functions that contains $G_r(x)$.

This then leads to Theorem \ref{thm_permrep}. To obtain the asymptote for $s_n(\Gamma)$ we use the fact that, when $\Gamma$ is a non-trivial free product, $h_n(\Gamma)$ grows so fast that most of the permutation representations of $\Gamma$ need to be transitive.

\subsection{Notes and references}

The first work on the number of permutation representations of a group goes back to the fifties of the previous century. In \cite{CHM}, Chowla, Herstein and Moore determined the asymptotic behavior of $h_n(\ZZ/2\ZZ)$ as $n\to \infty$. Their work was generalized by Moser and Wyman in \cite{MW1,MW2} to finite cyclic groups of prime order and by M\"uller in \cite{Mul2} to \emph{all} finite groups. M\"uller proved that for a given finite group $G$, we have
\[h_n(G) \sim R_G \; n^{-1/2}\; \exp\left(\sum_{d | \card{G}} \frac{s_d(G)}{d}\; n^{d/\card{G}}\right) (n!)^{1-1/\card{G}}\]
as $n\to\infty$, where $R_G>0$ is a constant only depending on $G$.

The subgroup growth of non-abelian free groups can be derived from Dixon's theorem on generating the symmetric group with random permutations \cite{Dixon_generating}. In \cite{Mul}, M\"uller determined the subgroup growth of free products of finite groups. In \cite{Mul2}, M\"uller also determined the asymptotic number of free subgroups in virtually free groups. More recently, Ciobanu and Kolpakov studied connections between such counts for $(\ZZ/2\ZZ)^{*3}$, the RACG associated to the graph on $3$ vertices with no edges, and other combinatorial objects \cite{CK}. The subgroup growth of surface groups was determined by M\"uller and Schlage-Puchta in \cite{MP1}, which was generalized to Fuchsian groups by Liebeck and Shalev in \cite{Liebeck_Shalev}. For a general introduction to the topic of subgroup growth we refer to the monograph by Lubotzky and Segal \cite{Lubotzky_Segal}.

There is also a vast body of literature available on the asymptotics of the coefficients in power series. In the case of the exponential generating functions of finite groups, M\"uller \cite{Mul2} uses a set of techniques developped by Hayman \cite{Hay} and Harris and Schoenfeld \cite{HS}. Our generating function also resembles the functions that were considered by Wright in \cite{Wri1,Wri2}. For more background on these techniques we refer the reader to \cite{Odl,FS}. 

\subsection*{Acknowledgement}
We are very grateful to an anonymous referee for proposing a simplified proof of Lemma \ref{lem_transactz2zr} and spotting a crucial mistake in a previous version of Proposition \ref{prp_closedform}.

\section{Preliminaries}

\subsection{Notation and set-up}
Given $n\in\NN$, set $[n] := \{1,\ldots, n\}$. Given a set $A$, $\sym(A)$ will denote the symmetric group on $A$ and $e\in \sym(A)$ will denote the trivial element. We will write $\sym_n = \sym([n])$.

We let
\[\inv(A)=\st{\pi\in\sym(A)}{\pi^2=e}\]
be the set of involutions in $\sym(A)$ and will again write $\inv_n = \inv([n])$. Given $k\leq n/2$, we write $\inv_{n,k}$ for the subset of $\inv_n$ consisting of involutions with $k$ $2$-cycles.

If $U\subset \sym_n$ is a subset, we will denote the centralizer of $U$ in $\sym_n$ by $Z(U)$ and the subgroup of $\sym_n$ generated by $U$ by $\langle U\rangle$. Moreover, if $V\subset \sym_n$ is another subset, we will write
\[Z_V(U) = Z(U)\cap V.\]

Given a finite graph $\mathcal{G}$ we will denote its vertex and edge sets by $V(\mathcal{G})$ and $E(\mathcal{G})$ respectively. $\cox(\mathcal{G})$ will denote the associated right-angled Coxeter group. That is
\[\cox(\mathcal{G}) = \langle \sigma_v,\; v\in V(\mathcal{G})|\: \sigma_v^2=e \; \forall v\in V(\mathcal{G}),\; [\sigma_v,\sigma_w]=e\; \forall \{v,w\}\in E(\mathcal{G}) \rangle.\]

\subsection{The exponential generating function}
Let $G$ be a finitely generated group. The exponential generating function for the seqeuence $h_n(G)$ is well known (see for instance \cite{Mul2}). Let us write
\[ F_G(x) = \sum_{n=0}^\infty \frac{h_n(G)}{n!} x^n\]
for this exponential generating function. It takes the following form:

\begin{lem}\label{lem_expprinc} Let $G$ be a finitely generated group. For all $x\in\CC$ we have
\[F_G(x) = \exp\left(\sum_{i=1}^{\infty} \frac{s_i(G)}{i} x^i\right).\]
\end{lem}

\section{Virtually cyclic right-angled Coxeter groups}

\subsection{Classification}
Now we specialize to virtually cyclic right-angled Coxeter groups. An example of such a group is $\cox(\mathcal{A}_r)$. Here, for $r\in\NN$, $\mathcal{A}_r$ is the graph obtained by taking the completee graph $\mathcal{K}_r$ and attaching two vertices that share an edge with each vertex in $\mathcal{K}_r$ but not with each other. For example, $\mathcal{A}_0$ consists of two vertices that do not share an edge and $\mathcal{A}_1$ is the line on three vertices.

We have
\[\cox(\mathcal{A}_r) \simeq (\ZZ/2\ZZ * \ZZ/2\ZZ) \times (\ZZ/2\ZZ)^r.\]
since the infinite dihedral group $(\ZZ/2\ZZ * \ZZ/2\ZZ)$ has an index $4$ subgroup isomorphic to $\ZZ$, $\cox(\mathcal{A}_r)$ is indeed virtually cyclic.

Our first observation is that all infinite virtually cyclic right-angled Coxeter groups are of form above:
\begin{lem}\label{lem_class} Let $\Gamma$ be an infinite virtually cyclic right-angled Coxeter group. Then there exists an $r\in\NN$ so that 
\[\Gamma = \cox(\mathcal{A}_r).\]
\end{lem}
\begin{proof} Suppose $\Gamma = \cox(\mathcal{G})$ for some finite graph $\mathcal{G}$. 

Label the vertices of $\mathcal{G}$ by $1,\ldots,s$. Since $\Gamma$ is assumed to be infinite, at least one pair of vertices of $\mathcal{G}$ does not share an edge. Let us suppose these are the vertices $1$ and $2$. We will argue that all other vertices need to be connected to both $1$ and $2$ and also to each other.

First suppose vertex $j>2$ is not connected to the vertex $1$. Then $\sigma_1\sigma_j$ and $\sigma_1\sigma_2$ are two infinite order elements. Moreover, there do not exist $m,n\in\NN$ so that $(\sigma_1\sigma_2)^m = (\sigma_1\sigma_j)^n$. This violates being virtually cyclic.

We conclude that vertices $3,\ldots,s$ are all connected to both $1$ and $2$. Now suppose there exists a pair of vertices $j,k>2$ that do not share an edge. Then $\sigma_1\sigma_2$ and $\sigma_j\sigma_k$ are a pair of infinite order elements without a common power. 

Putting the two observations together implies the lemma.
\end{proof}

\subsection{A closed formula}
Since the case of finite right-angled Coxeter groups is well understood, we will from hereon consider $\cox(\mathcal{A}_r)$. 

In order to obtain the asymptotes we are after, we are in need of a closed formula for the number of permutation representations of $\cox(\mathcal{A}_r)$. To this end, we first record two lemmas on the permutation representations of $(\ZZ/2\ZZ)^r$.

\begin{lem}\label{lem_transactz2zr}
\begin{itemize}
\item[(a)] Let $H<\sym_{2^r}$ be a transitive subgroup so that $H\simeq (\ZZ/2\ZZ)^r$. Then
\[Z_{\inv_{2^r}}(H) = H.\]
\item[(b)]
Suppose $l\leq r$ and let $\varphi \in\Hom\left((\ZZ/2\ZZ)^r, \sym_{2^l}\right)$ be so that $\varphi\left((\ZZ/2\ZZ)^r\right)$ acts on $[2^l]$ transitively. Then
\[ \varphi\left( (\ZZ/2\ZZ)^r\right) \simeq (\ZZ/2\ZZ)^l.\]
\end{itemize}
\end{lem}

\begin{proof} For item (a) we note that $Z_{\sym_{2^r}}(H)$ is an abelian group that acts transitively on $[2^r]$, from which it follows that it acts freely, which implies our claim.

For item (b), write $\varphi\left( (\ZZ/2\ZZ)^r\right) = Q<\sym_{2^l}$. Since $(\ZZ/2\ZZ)^r$ surjects onto $Q$, $Q$ is a finite abelian group in which every non-trivial element has order $2$. The fact that $Q$ acts on $[2^l]$ transitively, implies it has order $2^l$. The only groups that fit the description above is $(\ZZ/2\ZZ)^k$ for some $k\geq l$. Item (a) implies that $k=l$.
\end{proof}

We can split each $\varphi\in\Hom\left((\ZZ/2\ZZ)^r,\sym_n\right)$ into a product of homomorphisms into small symmetric groups according to the orbits of $\varphi\left((\ZZ/2\ZZ)^r\right)$. That is, we can identify $\varphi=\varphi_1\times \ldots \times \varphi_m$, where $\varphi_i\left((\ZZ/2\ZZ)^r\right)$ acts transitively on $[2^{l_i}]$ for some $0\leq i \leq r$. Lemma \ref{lem_transactz2zr}(b) above implies that we can identify $\varphi_i$ with a surjective homomorphism $(\ZZ/2\ZZ)^r \to (\ZZ/2\ZZ)^{l_i}$. With a slight abuse of notation, we will also call this homomorphism $\varphi_i$. Moreover, given a surjective homomorphism $\varphi:(\ZZ/2\ZZ)^r \to (\ZZ/2\ZZ)^l$ and $k\in\NN$, we will write 
\[\varphi^k= \varphi \times\ldots \times \varphi : (\ZZ/2\ZZ)^r \to \left((\ZZ/2\ZZ)^l \right)^k.\]

The next lemma is about how the centralizer $Z_{\sym_n}(\varphi\left( (\ZZ/2\ZZ)^r\right))$ depends on the decomposition of $\varphi$ as a product of homomorphisms.

\begin{lem}\label{lem_centdec}
Suppose $s\in\NN$ and $G_1,\ldots G_s<(\ZZ/2\ZZ)^r$ are all distinct subgroups. Moreover, let 
\[\varphi_{i,j}:(\ZZ/2\ZZ)^r \to (\ZZ/2\ZZ)^{l_i}\]
$j=1,\ldots m_i$, $i=1,\ldots,s$ be distinct surjective homomorphisms so that
\[\ker(\varphi_{i,j}) = G_i\]
for all $j=1,\ldots m_i$, $i=1,\ldots,s$. Let
\[\varphi:=  \bigtimes_{\substack{i=1,\ldots,s,\\ j=1,\ldots m_i}} \varphi_{i,j}^{k_{i,j}}:(\ZZ/2\ZZ)^r \to \bigtimes_{i=1,\ldots,s} \left(\ZZ/2\ZZ)^{l_i}\right)^{m_i} < \sym\left((\ZZ/2\ZZ)^{\sum_i m_i l_i}\right) \simeq \sym_{n} \]
where $n=2^{\sum_i m_i l_i}$. Then
\[Z_{\sym_n} (\varphi\left((\ZZ/2\ZZ)^r\right)) \simeq \bigtimes_{i=1}^s \left((\ZZ/2\ZZ)^{l_i} \wr \sym_{m_i} \right)\]
\end{lem}

\begin{proof} The lemma essentially consists of two claims: the fact that the centralizer decomposes as a product  and the fact that the factors take the form of a wreath product.

To see the product structure, note that when given $G < \sym_n$ and $\pi \in Z_{\sym_n}(G)$, $\pi$ acts on the orbits of $G$ on $[n]$. That is, for every $G$-orbit $A$ of cardinality $k$, there exists a $G$-orbit $B$ of cardinality $k$, so that
\[\pi(A) = B.\]
Because $\pi$ is a bijection, it can only permute orbits of the same size. This already implies that the centralizer splits as a product over the different orbit sizes. 

To get the full product decomposition, we need to understand which orbits can be permuted. We claim that two $\varphi\left((\ZZ/2\ZZ)^r\right)$-orbits $A, B \subset [n]$ of the same cardinality can be permuted by an element of $Z_{\sym_n}\left(\varphi\left((\ZZ/2\ZZ)^r\right)\right)$ if and only if the homomorphisms $\varphi_A,\varphi_B:(\ZZ/2\ZZ)^r \to (\ZZ/2\ZZ)^l$ they define have the same kernel. What we really need to show is that if $\ker(\varphi_A)\neq \ker(\varphi_B)$, then these orbits cannot be permuted. So, suppose $\sigma \in \sym_n$ so that 
\[\sigma(A)=B.\]
Because $\ker(\varphi_A)\neq \ker(\varphi_B)$ and both $\varphi_A$ and $\varphi_B$ are surjections onto $(\ZZ/2\ZZ)^l$ for the same $l$, we have that $\ker(\varphi_A)\smallsetminus \ker(\varphi_B) \neq \emptyset$. So, let $g\in \ker(\varphi_A)\smallsetminus \ker(\varphi_B) $ and let $a\in A$. Since $g\notin \ker(\varphi_B)$, there exists an $a\in A$ so that
\[\varphi(g)\left(\sigma(a)\right)\neq \sigma(a).\] 
However, because $\varphi(g)a=a$, we obtain
\[\sigma\left(\varphi(g) \sigma(a)\right)= \sigma(a).\] 
So, $\sigma$ cannot lie in  $Z_{\sym_n}\left(\varphi\left((\ZZ/2\ZZ)^r\right)\right)$. 

This implies that  $Z_{\sym_n}\left(\varphi\left((\ZZ/2\ZZ)^r\right)\right)$ indeed splits as a product according to the decomposition of $\varphi$ into homomorphisms with distinct kernels. Moreover, the fact that two orbits that define the same homomorphism \emph{can} be permuted, together with  Lemma \ref{lem_transactz2zr}(a) implies that the factos take the form we claim.
\end{proof}

In order to count the involutions in these centralizers, we record the following:

\begin{lem}\label{lem_invcent}
The number of involutions in $(\ZZ/2\ZZ)^l \wr \sym_k$ equals
\[  \sum_{r=0}^{\floor{k/2}} \frac{k!}{(k-2r)! \cdot r!} 2^{l\cdot k - (l+1)\cdot r}. \]
\end{lem}

\begin{proof}
In order to count the number of involutions, we need to understand the wreath product structure. We can write
\[(\ZZ/2\ZZ)^l \wr \sym_k = \st{(a, \pi)}{ a \in \bigg( (\ZZ/2\ZZ)^l \bigg)^k,\; \pi \in \sym_k}.\]
Multiplication of two elements is given by:
\[ ((a_1,\ldots, a_k),\; \pi) \cdot ((b_1,\ldots, b_k),\; \sigma) = ((a_{\sigma(1)} b_1, \ldots, a_{\sigma(k)}b_k), \pi\sigma).\]
This means that an element $(a,\pi)$ is an involution if and only if 
\[((a_{\pi(1)}a_1,\ldots, a_{\pi(k)} a_k),\; \pi^2) = ((e,\ldots,e),\;e),\]
which is equivalent to $\pi$ being an involution and that (using the fact that $(\ZZ/2\ZZ)^r$ consists entirely of involutions) $a_{\pi(i)} = a_i$ for all $i=1,\ldots, k$.

In order to count the number of involutions, note that $\sym_k$ contains $\binom{k}{2r} (2r)!!$ involutions with $k-2r$ fixed points. Here $(2r)!! = (2r-1)(2r-3)\cdots 3\cdot 1 = (2r)!/(2^r\cdot r!)$. After choosing the involution $\pi$ in $\sym_k$, there is a choice of one element of $(\ZZ/2\ZZ)^l$ left per orbit of $\pi$. So the number of involutions in $(\ZZ/2\ZZ)^l\wr \sym_k$ is
\[ \sum_{r=0}^{\floor{k/2}} \binom{k}{2r} (2r)!! \; 2^{l(k-r)} = \sum_{r=0}^{\floor{k/2}} \frac{k!}{(k-2r)! \cdot r!} 2^{l\cdot k - (l+1)\cdot r}. \]
\end{proof}

This now gives us the following expression for the number of permutation representations of $\cox(\mathcal{A}_r)$:
\begin{prp}\label{prp_closedform} Let $r\in\NN$ and $\Gamma=\cox(\mathcal{A}_r)$. Then
\[h_n(\Gamma) = n! \sum_{\substack{i_0,\ldots,i_r \\ \sum_j i_j 2^j=n}} \sum_{\substack{k_{j,1},\ldots, k_{j,s_{2^j}},\\ \sum_m k_{j,m} = i_j } }  \prod_{j=0}^r 2^{j\;i_j} \prod_{m=1}^{s_{2^j}} k_{j,m} ! \;  \left(\sum_{l_{j,m}=0}^{\floor{k_{j,m}/2}} \frac{1}{(k_{j,m}-2l_{j,m})! \cdot l_{j,m}!} 2^{ - (j+1)\cdot l_{j,m}}\right)^2,\]
for all $n\in\NN$, where
\[s_{2^j} = s_{2^j}((\ZZ/2\ZZ)^r) = \frac{\prod_{l=0}^{j-1} (2^r-2^l)}{\prod_{l=0}^{j-1} (2^j-2^l)}. \]
for all $j=0,\ldots, r$.
\end{prp}

\begin{proof}
We have
\[h_n(\Gamma) = \sum_{\varphi \; \in \; \Hom\left((\ZZ/2\ZZ)^r,\sym_n\right)} \card{Z_{\inv_n}\left(\varphi\left((\ZZ/2\ZZ)^r\right)\right)}^2\]
Lemma \ref{lem_centdec} tells us that the cardinality $ \card{Z_{\inv_n}\left(\varphi\left((\ZZ/2\ZZ)^r\right)\right)}$ only depends on the way $\varphi$ decomposes into homomorphisms with distinct kernels.

So
\[h_n(\Gamma) = \sum_{\substack{i_0,\ldots,i_r \\ \sum_j i_j 2^j=n}}\; \sum_{\substack{k_{j,1},\ldots, k_{j,s_{2^j}},\\ \sum_m k_{j,m} = i_j } } h_{i,k}\left((\ZZ/2\ZZ)^r\right) \cdot  z_{i,k}^2, \]
where for each $j$, labeling the subgroups of $(\ZZ/2\ZZ)^r$ of index $2^j$ by 
\[G_{j,1},\ldots, G_{j,s^{2^j}},\] 
we write $h_{i,k}\left((\ZZ/2\ZZ)^r\right)$ for the number of homomorphisms $(\ZZ/2\ZZ)^r\to \sym_n$ so that
\begin{itemize}
\item $\varphi\left((\ZZ/2\ZZ)^r\right)$ has $i_j$ orbits of size $2^j$ for $j=0,\ldots,r$ and
\item for $j=0,\ldots,r$, $k_{j,m}$ of the orbits correspond to a map $\rho:(\ZZ/2\ZZ)^r \to (\ZZ/2\ZZ)^j$ with kernel $G_j$.
\end{itemize}
Moreover, $z_{i,k}$ is the number of involutions in the centralizer $Z_{\sym_n}\left(\varphi\left((\ZZ/2\ZZ)^r \right)\right)$ for any such homomorphism $\varphi$.

In order to count $h_{i,k}\left((\ZZ/2\ZZ)^r\right)$, note that there are
\[n! \prod_{j=0}^r \frac{1}{i_j!((2^j)!)^{i_j}}\]
ways to partition $[n]$ into orbits whose sizes are given by $i$, and 
\[\prod_{j=0}^r i_j! \prod_{m=1}^{s_{2^j}} \frac{1}{k_{j,m}!} \]
to choose which groups belong to which orbits.

Furthermore, there are $(2^j-1)!$ different tansitive homomorphisms with kernel $G_{j,m}$ for all $j=0,\ldots, r$, $m=1,\ldots, s_{2^j}$. So we obtain
\[h_{i,k}\left((\ZZ/2\ZZ)^r\right) =  n! \prod_{j=0}^r \left( \frac{1}{((2^j)!)^{i_j}} \prod_{m=1}^{s_{2^j}} ((2^j-1)!)^{k_{j,m}} \frac{1}{k_{j,m}!} \right) = n! \prod_{j=0}^r \frac{1}{2^{j\;i_j}}  \prod_{m=1}^{s_{2^j}}\frac{1}{k_{j,m}!},\]
where we used that $\sum_m k_{j,m} = i_j$. 

Lemmas \ref{lem_centdec} and \ref{lem_invcent} imply that
\begin{multline*}
z_{i,k} = \prod_{j=0}^r \prod_{m=1}^{s_{2^j}} \sum_{l_{j,m}=0}^{\floor{k_{j,m}/2}} \frac{k_{j,m} !}{(k-2l_{j,m})! \cdot l_{j,m}!} 2^{j\cdot k_{j,m} - (j+1)\cdot l_{j,m}} \\ = \prod_{j=0}^r  2^{j\; i_j} \prod_{m=1}^{s_{2^j}} \sum_{l_{j,m}=0}^{\floor{k_{j,m}/2}} \frac{k_{j,m} !}{(k-2l_{j,m})! \cdot l_{j,m}!} 2^{ - (j+1)\cdot l_{j,m}}\end{multline*}

Putting the two together, we obtain the formula we claimed.
\end{proof}

\section{An exponential generating function}

Our next step is to compute the exponential generating function for the sequence\linebreak $(h_n(\cox(\mathcal{A}_r))_n$. That is, we define
\[G_r(x) = \sum_{n=0}^\infty \frac{h_n(\cox(\mathcal{A}_r))}{n!}\; x^n.\]

\begin{thm}\label{thm_genfn} Let $r\in \NN$. Then
\[G_r(x) = \prod_{j=0}^r \left( \left(1- x^{2^{j+1}}\right)^{-s_{2^j}/2} \exp\left(-2^j\; s_{2^j}+\frac{2^j\; s_{2^j}}{1-x^{2^j}}\right)\right)\]
where $s_{2^j}=s_{2^j}((\ZZ/2\ZZ)^r)$.
\end{thm}

\begin{proof} Proposition \ref{prp_closedform} tells us that
\begin{multline*}
G_r(x)  = 
 \sum_{i_0,\ldots,i_r = 0}^\infty x^{\sum_j i_j 2^j}  \sum_{\substack{k_{j,1},\ldots, k_{j,s_{2^j}},\\ \sum_m k_{j,m} = i_j } }  \prod_{j=0}^r 2^{j\;i_j}  \prod_{m=1}^{s_{2^j}} k_{j,m} ! \\
 \cdot \left(\sum_{l_{j,m}=0}^{\floor{k_{j,m}/2}} \frac{1}{(k_{j,m}-2l_{j,m})! \cdot l_{j,m}!} 2^{ - (j+1)\cdot l_{j,m}}\right)^2 \\
  =  
  \prod_{j=0}^{r} \sum_{i_j=0}^\infty \left(2^j x^{2^j}\right)^{i_j} \sum_{\substack{k_{j,1},\ldots, k_{j,s_{2^j}},\\ \sum_m k_{j,m} = i_j } }    \prod_{m=1}^{s_{2^j}} k_{j,m} ! \left(\sum_{l_{j,m}=0}^{\floor{k_{j,m}/2}} \frac{1}{(k_{j,m}-2l_{j,m})! \cdot l_{j,m}!} 2^{ - (j+1)\cdot l_{j,m}}\right)^2  \\
   =  
  \prod_{j=0}^{r}   \sum_{k_{j,1},\ldots, k_{j,s_{2^j}} =0 }^\infty   \left(2^j x^{2^j}\right)^{\sum_m k_{j,m}} \prod_{m=1}^{s_{2^j}} k_{j,m} !\left(\sum_{l_{j,m}=0}^{\floor{k_{j,m}/2}} \frac{1}{(k_{j,m}-2l_{j,m})! \cdot l_{j,m}!} 2^{ - (j+1)\cdot l_{j,m}}\right)^2 \\
    =  
   \prod_{j=0}^{r}  \left(\sum_{k_j=0}^\infty \left(2^j x^{2^j}\right)^{k_j} k_j ! \; \left(\sum_{l_j=0}^{\floor{k_j/2}} \frac{1}{(k_j-2l_j)! \cdot l_j!} 2^{ - (j+1)\cdot l_j}\right)^2 \right)^{s_{2^j}} 
\end{multline*}
which leads us to define two sequences 
\[b_{j,k}:=\sum_{l=0}^{\floor{k/2}} \frac{2^{-(j+1)\;l}}{(k-2l)!\;l!} \;\; \text{and}\;\; b_{j,k}^{(2)}:= k!\;b_{j,k}^2 \]
for all $i,j\in\NN$ and the corresponding generating functions
\[F_j(x) = \sum_{k=0}^\infty b_{j,k}\; x^k \;\; \text{and} \;\; F_j^{(2)}(x) = \sum_{k=0}^\infty b_{j,k}^{(2)}\; x^k\]
so that
\[G_r(x) = \prod_{j=0} \left(F_j^{(2)}(2^j x^{2^j})\right)^{s_{2^j}} . \]

We will now prove the theorem by first determining $F_j(x)$, which leads to a recurrence for the sequence $(b_{j,k})_k$. From that, we will derive a recurrrence for the sequence $(b_{j,k}^{(2)})_i$, which in turn leads to an ODE for $F_j^{(2)}(x)$. The solution to this ODE then gives us $F_j^{(2)}(x)$ and hence $G_r(x)$.

We have 
\[F_j(x) = \sum_{k=0}^\infty  \sum_{l = 0}^{\floor{k/2}} \frac{2^{-(j+1)l}}{(k-2l)!\; l!} x^k =  \sum_{l=0}^\infty  \sum_{k = 2l}^{\infty} \frac{2^{-(j+1)l}}{(k-2l)!\; k!} x^k .\]
Changing the index in the innermost sum, we obtain
\[F_j(x) =  \sum_{l=0}^\infty  \sum_{k = 0}^{\infty} \frac{2^{-(j+1)l}}{k!\; l!} x^{k+2l}  = \exp(x + 2^{-(j+1)} x^2).\]
Using that $d F_j(x)/dx= (1+2^{-j}\;x) F_j(x)$ and equating coefficients, we get the recurrence
\[ (k+1)\; b_{j,k+1} = b_{j,k} + 2^{-j}\; b_{j,k-1}.\]
for all $k\geq 0$.

The recurrence for the sequence $(b_{j,k})_k$ implies that
\begin{multline*}
b_{j,k}^{(2)} = \frac{(k-1)!}{k}\;(b_{j,k-1}+2^{-j}\;b_{j,k-2})^2 \\
 = \frac{1}{k}\;b_{j,k-1}^{(2)} + \frac{2^{-2j} (k-1)}{k}\; b_{j,k-2}^{(2)} + \frac{2^{-j+1}}{k}\;(k-1)!\; b_{j,k-1}\;b_{j,k-2}.  
\end{multline*}
To get the recurrence we are after, we need to compute the cross terms in the above. Again using the recurrence for the sequence $(b_{j,k})_k$, we obtain
\[ (k-1)!\;b_{j,k-1}\;b_{j,k-2} = b_{j,k-2}^{(2)} + 2^{-j}\; (k-2)!\; b_{j,k-2}\;b_{j,k-3}. \]
Hence
\[ (k-1)!\;b_{j,k-1}\;b_{j,k-2} = \sum_{l=2}^k 2^{-(l-2)j}\; b_{j,k-l}^{(2)} \]
and
\[b_{j,k}^{(2)} = \frac{1}{k}\;b_{j,k-1}^{(2)} + 2^{-2j}\frac{k-1}{k}\; b_{j,k-2}^{(2)} + 2^{-j+1}\frac{1}{k}\;\sum_{l=2}^k 2^{-j\;(l-2)}\; b_{j,k-l}^{(2)},  \]
or equivalently
\[k\; \left(b_{j,k}^{(2)} - 2^{-2j}\; b_{j,k-2}^{(2)} \right) = b_{j,k-1}^{(2)} - 2^{-2j} \; b_{j,k-2}^{(2)} + 2^{j+1}\;\sum_{l=2}^k 2^{-j\;l}\; b_{j,k-l}^{(2)}.\]

Using the recurrence, we obtain
\begin{eqnarray*}
\frac{d}{dx} \left(F_j^{(2)}(x) - 2^{-2j}\; x^2\; F_j^{(2)}(x) \right) & = & \sum_{k=0}^\infty k \left(b_{j,k}^{(2)} - 2^{-2j}\; b_{j,k-2}^{(2)} \right) \; x^{k-1} \\
& = & \sum_{k=0}^\infty \left(b_{j,k-1}^{(2)} - 2^{-2j} \; b_{j,k-2}^{(2)} + 2^{j+1}\;\sum_{l=2}^k 2^{-j\;l}\; b_{j,k-l}^{(2)}\right)\;x^{k-1}.
\end{eqnarray*}
We have 
\[ \sum_{k=0}^\infty (b_{j,k-1}^{(2)} - 2^{-2j}  b_{j,k-2}^{(2)})\; x^{k-1} = (1-2^{-2j} x)\;F_j^{(2)}(x) \]
and
\[ \sum_{k=0}^\infty \sum_{l=2}^k 2^{-j\;l}\; b_{j,k-l}^{(2)} \; x^{k-1} = \sum_{l=2}^\infty \sum_{k=0}^\infty 2^{-j\;l}\; b_{j,k}^{(2)} \; x^{k+l-1} = \frac{2^{-2j}\;x}{1-2^{-j}\;x}\; F_j^{(2)}(x). \]
Because
\[ \frac{d}{dx} \left(F_{j}^{(2)}(x) - 2^{-2j}\; x^2\; F_{j}^{(2)}(x) \right) = (1-2^{-2j}\; x^2) \; \frac{d}{dx}F_j^{(2)}(x) - 2^{-2j+1}\;x\; F_j^{(2)}(x),\]
our generating function satisfies the following ODE:
\[(1-2^{-2j}\; x^2) \; \frac{d}{dx}F_j^{(2)}(x) = (1+2^{-2j}\;x)\; F_j^{(2)}(x) + 2^{-j+1}\frac{x}{1-2^{-j}\;x}\; F_j^{(2)}(x), \]
which is equivalent to
\[ \frac{d}{dx}F_j^{(2)}(x) = \frac{1 + (2^{-2j} + 2^{-j})\;x - 2^{-3j}\; x^2}{(1-2^{-j}\;x)^2\;(1+2^{-j}\;x)}\; F_j^{(2)}(x),\]
which leads to 
%
%
%
\[ F_j^{(2)}(x) = a \; \frac{1}{\sqrt{1-2^{-2j}  x^2 }} \exp\left(\frac{-1}{2^{-2j}\;x-2^{-j}}\right), \]
where $a\in\CC$ is some constant. Filling this in for $G_r(x)$ leads to
\[G_r(x) = a' \prod_{j=0}^r \left( \left(1- x^{2^{j+1}}\right)^{-s_{2^j}/2} \exp\left(\frac{2^j\; s_{2^j}}{1-x^{2^j}}\right)\right)\]
for some constant $a'\in\CC$. Equating the constant coefficient gives 
\[a'=\exp\left(-\sum_{j=0}^r 2^j s_{2^j} \right) .\]
\end{proof}

\section{Immediate consequences}

Before we turn to the asymptotic behavior of the sequence $(h_n(\cox(\mathcal{A}_r)))_n$, we derive two immediate consequences to Theorem \ref{thm_genfn}: a closed formula for the number index $n$ subgroups of $\cox(\mathcal{A}_r)$ and a recurrence for the sequence $(h_n(\cox(\mathcal{A}_r)))_n$.

\subsection{Subgroups of virtually cyclic Coxeter groups}

Theorem \ref{thm_genfn} allows us to derive a closed form  for the number of subgroups of a virtually cyclic Coxeter group. We have:

\begin{cor}\label{cor_subgroupcount} Let $r \in \NN$ and let
\[\Gamma = \cox(\mathcal{A}_r).\]
Then
\[s_n(\Gamma) = n\; \left(1+ \sum_{\substack{ 0< j \leq r \\ s.t. \;  2^j | n}} 2^j\; s_{2^j} \right) + \sum_{\substack{ 0\leq  j \leq r \\ s.t. \;  2^{j+1} | n}} 2^j\; s_{2^j},\]
for all $n\in \NN$.
\end{cor}

\begin{proof}
Lemma \ref{lem_expprinc} implies that
\[\sum_{n = 1}^\infty \frac{s_n(\Gamma)}{n}\; x^n = \log\left(\sum_{n=0}^\infty \frac{h_n(\Gamma)}{n!}\;x^n \right)\]
Using Theorem \ref{thm_genfn}, this means that
\[\sum_{n = 1}^\infty \frac{s_n(\Gamma)}{n}\; x^n =  \sum_{j=0}^r s_{2^j}\cdot \left(- 2^j-\frac{1}{2}\log\left(1- x^{2^{j+1}}\right) + \frac{2^j}{1- x^{2^j}} \right).\]
Using the Taylor expansions for $-\log(1-x)$ and $1/(1-x)$ at $x=0$, we obtain
\[\sum_{n = 1}^\infty \frac{s_n(\Gamma)}{n}\; x^n =  \sum_{j=0}^r s_{2^j} \sum_{k=1}^\infty \frac{1}{2}\frac{1}{k} x^{k\;2^{j+1}} + 2^j\; x^{k\;2^j}.\]
This means that
\[s_n(\Gamma) = n\; \left(1+\frac{1}{2}\sum_{\substack{ 0\leq  j \leq r \\ s.t. \;  2^{j+1} | n}} \frac{2^{j+1}}{n}\;s_{2^j}+ \sum_{\substack{ 0< j \leq r \\ s.t. \;  2^j | n}} 2^j\; s_{2^j} \right),\]
which gives the corollary.
\end{proof}

\subsection{A recurrence}

As a corollary to Theorem \ref{thm_genfn}, or rather its proof, we obtain the following recurrence for the sequence $\left(h_n(\mathcal{A}_r)\right)_n$:
\begin{cor}\label{cor_recurrence} Let $r\in\NN$ and write
\[h_{n,r} = h_n(\cox(\mathcal{A}_r)).\]
Then
\begin{eqnarray*}
 \frac{h_{n+1,r}}{n!} & = & - \sum_{\substack{\varepsilon \in I_r \setminus \{0\} \\ \abs{\varepsilon} \leq n}} \cdot \frac{h_{n-\abs{\varepsilon}+1,r}}{(n-\abs{\varepsilon})!}\; \prod_{j=0}^r (-1)^{\varepsilon_{j,1}+\varepsilon_{j,2}} 
 \\
 & & + \sum_{j=0}^r \sum_{\substack{\varepsilon \in I_r \\ \varepsilon_{j,1}=\varepsilon_{j,2} \\ =\varepsilon_{j,3}=0}} \frac{h_{n+1-2^j-\abs{\varepsilon},r}}{(n+1-2^j-\abs{\varepsilon})!} \cdot s_{2^j}\cdot 2^{2j} \; \prod_{k\neq j} (-1)^{\varepsilon_{k,1}+\varepsilon_{k,2}} 
 \\
  & & + \sum_{j=0}^r \sum_{\substack{\varepsilon \in I_r \\ \varepsilon_{j,1}=\varepsilon_{j,2}\\ =\varepsilon_{j,3}=0}}  \frac{h_{n+1-2^{j+1}-\abs{\varepsilon},r}}{(n+1-2^{j+1}-\abs{\varepsilon})!} \cdot s_{2^j}\cdot (2^j+2^{2j})\;  \prod_{k\neq j} (-1)^{\varepsilon_{k,1}+\varepsilon_{k,2}}
   \\
  & & - \sum_{j=0}^r  \sum_{\substack{\varepsilon \in I_r \\ \varepsilon_{j,1}=\varepsilon_{j,2}\\ =\varepsilon_{j,3}=0}}  \frac{h_{n+1-3\cdot 2^j-\abs{\varepsilon},r}}{(n+1-3\cdot 2^j-\abs{\varepsilon})!} \cdot s_{2^j}\cdot 2^j \; \prod_{k\neq j} (-1)^{\varepsilon_{k,1}+\varepsilon_{k,2}}
\end{eqnarray*}
for all $n\geq 0$, with the initial conditions
\[h_{0,r} = 1\;\; \text{and} \;\;h_{m,r} = 0 \;\text{for all}\; m<0,\]
where $I_r = \left(\{0,1\}^{r+1}\right)^3$, $\abs{\varepsilon} = \sum_{j=0}^r (\varepsilon_{j,1}+\varepsilon_{j,2}+\varepsilon_{j,3})\;2^j$ for all $\varepsilon\in I_r$ and $s_{2^j} = s_{2^j}((\ZZ/2\ZZ)^r)$.
\end{cor}

\begin{proof} Recall the differential equation for the functions $F_j^{(2)}$ defined in the proof of Theorem \ref{thm_genfn}:
\[ \frac{d}{dx}F_j^{(2)}(x) = \frac{1 + (2^{-2j} + 2^{-j})\;x - 2^{-3j}\; x^2}{(1-2^{-j}\;x)^2\;(1+2^{-j}\;x)}\; F_j^{(2)}(x),\]
for $j=0,\ldots,r$. Using the fact that
\[G_r(x) = \prod_{j=0} \left(F_j^{(2)}(2^j x^{2^j})\right)^{s_{2^j}} . \]
we obtain
\[\frac{dG_r(x)}{dx} = \sum_{j=0}^r s_{2^j}\;2^{2j}\;x^{2^j-1}\;\frac{1 + (2^{-j} + 1)\;x^{2^j} - \;2^{-j}\; x^{2^{j+1}}}{(1-x^{2^j})^2\;(1+x^{2^j})} G_r(x), \]
or equivalently
\begin{eqnarray*}
\frac{dG_r(x)}{dx}\prod_{j=0}^r (1-x^{2^j})^2\;(1+x^{2^j})  &  = &
 \sum_{j=0}^r \bigg\{\Big(2^{2j} x^{2^j-1} + (2^j + 2^{2j})\;x^{2^{j+1}-1} - 2^j\; x^{3\cdot 2^j-1}\Big) \\
& &   
 \cdot s_{2^j} \cdot \prod_{k\neq j} (1-x^{2^k})^2\;(1+x^{2^k})\bigg\}\cdot G_r(x) .
\end{eqnarray*}
So,
\begin{eqnarray*}
\sum_{n=0}^\infty \frac{h_{n+1,r}\; x^n}{n!} \prod_{j=0}^r (1-x^{2^j})^2\;(1+x^{2^j}) & = &  
 \sum_{j=0}^r \bigg\{\sum_{n=0}^\infty \frac{h_{n,r}\;x^n}{n!}  \cdot s_{2^j} \cdot \Big( 2^{2j} x^{2^j-1} - 2^j\; x^{3\cdot 2^j-1}\\
 & & + (2^j + 2^{2j})\;x^{2^{j+1}-1} \Big) \cdot \prod_{k\neq j} (1-x^{2^k})^2\;(1+x^{2^k})\bigg\}.
\end{eqnarray*}
Equating the coefficients of $x^n$ now gives the corollary.
\end{proof}

\section{Asymptotics} The goal of this section is to prove Theorems \ref{thm_permrep} and \ref{thm_subgrps}: the asymptotes for the number of permutation representations and the number of subgroups of a free product of virtually cyclic right-angled Coxeter groups. The largest part of the section is taken up by the proof of Theorem \ref{thm_asymp}, which gives an asymptotic expression for the coefficients in functions of the form of $G_r(x)$.

\subsection{The asymptotics of coefficients of power series}\label{sec_genfns}

There are many methods available to determine the asymptotics of the coefficients of a given power series $F(z)$. We will use Hayman's techniques from \cite{Hay}. Hayman's results hold for functions that are by now called H-admissible functions. In the following definition we will write
\[\DD_R = \st{z\in\CC}{\abs{z}<R},\]
for $R>0$.
\begin{dff}\label{dff_hadmissible} Let $R>0$ and let $F:\DD_R\to \CC$. Define $A,B:[0,R)\to \RR$ by
\begin{equation}\label{eq_functionsH}
A(\rho) = \rho\;\frac{1}{F(\rho)}\; \frac{d}{d\rho} F(\rho) \;\;\text{and} \;\; B(\rho)= \rho\;\frac{d}{d\rho} A(\rho),
\end{equation}
Then $F$ is called H-admissible if there exists a function $\delta;[0,R]\to (0,\pi)$ so that $F$ satisfies the following conditions
\begin{itemize}
\item[(H1)] As $\rho\to R$, 
\[F(\rho e^{i\theta}) \sim F(\rho) \; \exp\left(i\theta A(\rho) - \frac{1}{2}\theta^2\;B(\rho)\right),\] 
uniformly for $\abs{\theta} \leq \delta(\rho)$ 
\item[(H2)] As $\rho\to R$, 
\[F(\rho e^{i\theta}) = o\left(\frac{F(\rho)}{\sqrt{B(\rho)}} \right), \]
uniformly for $\delta(\rho) \leq \abs{\theta} \leq \pi$ 
\item[(H3)] As $\rho\to R$, 
\[B(\rho) \to \infty.\]
\end{itemize}
\end{dff}

Hayman \cite[Theorem 1]{Hay} proved that
\begin{thm}\label{thm_Hayman} Let $R>0$ and let $F:\DD_R\to \CC$ be given by
\[ F(z) = \sum_{n=0}^\infty f_n\; z^n\]
and suppose $F$ is H-admissible. Then
\[f_n \rho^n =\frac{F(\rho)}{\sqrt{2\pi B(\rho)}}\exp\left(-\frac{1}{4}\frac{\left(A(\rho)-n\right)^2}{B(\rho)}\right)  \; \big(1+o(1)\big) \]
as $\rho\to R$ uniformly for all natural numbers $n\in\NN$.
\end{thm}

Our strategy now consists of three steps. First, we use Hayman's results to get a uniform estimate on the coefficients of $F$ in terms of a radius $\rho$ and two functions $A(\rho)$ and $B(\rho)$  (Proposition \ref{prp_coeff}). The standard trick after this (that already appears in Hayman's paper) is to simplify the estimates by finding a sequence $(\rho_n)_n$ that solves the equation $A(\rho_n)=n$. In order to get good estimates on such a sequence, we first prove a lemma (Lemma \ref{lem_roots}) on the solutions to polynomial equations. This relies on Newton's method. Finally, we put our estimates together, which gives us the asymptotic we are after (Theorem \ref{thm_asymp}). 

In this section, we will just present the results. The proofs of these results, that are independent of the rest of the material, will be postponed to later sections.

Theorem \ref{thm_Hayman} leads to the following estimate:

\begin{prp}\label{prp_coeff} For $r\in\NN_{\geq 1}$, let $b_{1,j},b_{2,j} >0$ and $k_j\in\NN$ for $j=0,\ldots ,r$, so that $k_0=1$ and 
\[k_0 < k_j \;\;\text{for all}\; j\in \{1,2,3,\ldots,r\}.\] 
Moreover, let
\[F : \DD_1 \to \CC\]
be defined by
\[F(z) = \sum_{n\geq 0} f_n \; z^n = \prod_{j=0}^r \left(1-z^{2k_j}\right)^{-b_{1,j}}\exp\left( \frac{b_{2,j}}{1-z^{k_j}} \right). \]
Then
\[f_n \rho^n =\frac{F(\rho)}{\sqrt{2\pi B(\rho)}}\exp\left(-\frac{1}{4}\frac{\left(A(\rho)-n\right)^2}{B(\rho)}\right)  \; \big(1+o(1)\big) \]
as $\rho\to 1$ uniformly for all natural numbers $n\in\NN$, where
\[
A(\rho) = \sum_{j=0}^r b_{1,j}\;2k_j \;\frac{\rho^{2k_j} }{1-\rho^{2k_j} } + b_{2,j} \;k_j\; \frac{\rho^{k_j}}{(1-\rho^{k_j})^2}\]
and
\[
B(\rho) =  \sum_{j=0}^r 4\; b_{1,j}\; k_j^2\; \left(\frac{\rho^{2k_j} }{1-\rho^{2k_j} } + \left(\frac{\rho^{2k_j} }{1-\rho^{2k_j}}\right)^2 \right)   +\; b_{2,j} \;k_j^2\; \left( \frac{\rho^{k_j}}{(1-\rho^{k_j} )^2}  + 2\;\frac{\rho^{2k_j} }{(1- \rho^{k_j})^3}\right) .
\]
\end{prp} 

The proof of this proposition, which essentially consists of proving that $F$ is H-admissible, can be found in Section \ref{sec_proof1}. 

In order to control the function $A(\rho)$, we will need the following lemma:
\begin{lem}\label{lem_roots} Let $r\in\NN$, $\alpha_j, \beta_j \in (0,\infty)$ and $k_j\in \NN_{\geq 1}$, for $j=0,\ldots,r$. Consider for $t\in (0,1)$ the equation in $y$ given by
\[t\;\sum_{j=0}^r \alpha_j \; \frac{y^{2k_j}}{1-y^{2k_j}} + \beta_j\; \frac{y^{k_j}}{(1-y^{k_j})^2} = 1.\]
This equation has a solution $y=y(t)$ that satisfies
\[y (t) = 1 - c_1 t^{1/2} + c_2 t +\bigO{t^\gamma} \]
as $t\to 0$ for some $\gamma \in \QQ_{>1}$, where the constant $c_1>0$ is given by
\[c_1 = \sqrt{\sum_{j=0}^r\frac{\beta_j}{k_j^2}}.\]
\end{lem}

The proof of this lemma, which is an application of the Newton-Puiseux method, can be found in Section \ref{sec_proof2}. We also note that this method also allows us to determine $c_2$, it will however not be needed in what follows, it disappears in the proof of the following:
\begin{thm}\label{thm_asymp} Let $(f_n)_{n\in \NN}$ be as above. Then
\[
f_n   \sim  C_1(b,k,r)\cdot n^{-3/4 + \sum_{j=0}^r b_{1,j}/2} \exp\left(C_2(b,k,r)\cdot\sqrt{n}\right) ,
\]
as $n\to\infty$, where
\[
C_1(b,k,r) =  \frac{ \left(\sum_{j=0}^r \frac{b_{2,j}}{k_j}\right)^{1/4} }{\sqrt{4\pi}}\exp\left(\sum_{j=0}^r \frac{b_{2,j}}{2k_j} + \sum_{j=0}^r \frac{k_j-1}{2k_j}\right) \prod_{j=0}^r \left(\frac{1}{2k_j\; \sqrt{\sum_{j=0}^r \frac{b_{2,j}}{k_j}}} \right)^{b_{1,j}} 
\]
and
\[C_2(b,k,r) = 2\sqrt{\sum_{j=0}^r \frac{b_{2,j}}{k_j}}.\]
\end{thm}

We will prove this theorem in Section \ref{sec_proof3}.

\subsection{The number of permutation representations}\label{sec_permreps}

Using Theorem \ref{thm_asymp} we can now determine the asymptotic number of permutation representations of a free product of virtually cyclic right-angled Coxeter groups. Before we state it, we define some constants depending on a given Coxeter group.

\begin{dff}\label{def_csts}
Let $r_1,\ldots,r_m \in \NN$ and
\[\Gamma = \Asterisk_{l=1}^m \cox(\mathcal{A}_{r_l}) .\]
Then we define
\[
A_\Gamma = \prod_{l=1}^m  \frac{ (s_{\mathrm{tot}}(r_l))^{1/4-s_{\mathrm{tot}(r_l)}/4} }{\sqrt{4\pi}\; 2^{\sum_{j=0}^{r_l} (j+1)s_{2^j}/2}} \;
 \exp\left(-\sum_{j=0}^{r_l} 2^j s_{2^j}+\frac{1}{2}s_{\mathrm{tot}}(r_l) + \frac{1}{2}r_l -\frac{1}{2} + \frac{1}{2^{r_l+1}} \right)  ,
\]
\[B_\Gamma =2\sum_{l=0}^m \sqrt{s_{\mathrm{tot}}(r_l)}\]
and
\[C_\Gamma =-3m/4+\sum_{l=0}^m s_{\mathrm{tot}}(r_l)/4.\]
Here $s_{\mathrm{tot}}(r)$ denotes the total number of subgroups of $(\ZZ/2\ZZ)^r$. In other words,
\[s_{\mathrm{tot}}(r) = \sum_{j=0}^r s_{2^j}. \]
\end{dff}

This now allows us to write down the asymptote for $h_n(\Gamma)$:

\begin{thm}\label{thm_permrep} Let $r_1,\ldots,r_m \in \NN$ and
\[\Gamma = \Asterisk_{l=1}^m \cox(\mathcal{A}_{r_l}) .\]
Then 
\[ h_n(\Gamma) \sim A_\Gamma \;n^{C_\Gamma}\; \exp(B_\Gamma\;\sqrt{n}) \; n!^m.\]
as $n\to\infty$.
\end{thm}

\begin{proof} For the case $m=1$ and $r_1=0$, this follows directly from the classical result by Chowla, Herstein and Moore \cite{CHM}. However, our methods also apply.

Since 
\[h_n(\Gamma) = \prod_{l=1}^m h_n(\cox(\mathcal{A}_{r_l})),\]
we can determine the asymptote for each factor independently. In other words, all we need to determine is the asymptote for $h_n(\cox(\mathcal{A}_r))$

In the language of Theorem \ref{thm_asymp} and using Theorem \ref{thm_genfn}, we have
\[b_{1,j} = s_{2^j}/2, \; \; b_{2,j} = 2^js_{s^j}\; \; \text{and} \; \; k_j = 2^j  \]
So, filling in the constants, we get
\begin{multline*}
C_1(b,k,r) =  \frac{ (s_{\mathrm{tot}}(r))^{1/4} }{\sqrt{4\pi}}\exp\left(\frac{1}{2}s_{\mathrm{tot}}(r) + \sum_{j=0}^r \frac{2^j-1}{2^{j+1}} \right) \prod_{j=0}^r \left(\frac{1}{2^{j+1}\; \sqrt{s_{\mathrm{tot}}(r)}} \right)^{s_{2^j}/2} \\
=  \frac{ (s_{\mathrm{tot}}(r))^{1/4} }{\sqrt{4\pi}}\exp\left(\frac{1}{2}s_{\mathrm{tot}}(r) + \frac{1}{2}r -\frac{1}{2} + \frac{1}{2^{r+1}}\right) \prod_{j=0}^r \left(\frac{1}{2^{j+1}\; \sqrt{s_{\mathrm{tot}}(r)}} \right)^{s_{2^j}/2} 
\end{multline*}
and
\[C_2(b,k,r) = 2\sqrt{\sum_{j=0}^r s_{2^j}} = 2\sqrt{s_{\mathrm{tot}}(r)}.\]
So, using Theorem \ref{thm_asymp} and the extra multiplicative factor from Theorem \ref{thm_genfn}, we obtain 
\begin{multline*}
h_n(\cox(\mathcal{A}_r)) \sim  \frac{ (s_{\mathrm{tot}}(r))^{1/4} }{\sqrt{4\pi}}
\cdot \exp\left(-\sum_{j=0}^r 2^j s_{2^j}+\frac{1}{2}s_{\mathrm{tot}}(r) + \frac{1}{2}r -\frac{3}{2} + \frac{1}{2^r} \right)   \\
\cdot\prod_{j=0}^r \left(\frac{1}{2^{j+1}\; \sqrt{s_{\mathrm{tot}}(r)}} \right)^{s_{2^j}/2}  \cdot n^{s_{\mathrm{tot}}(r)/4-3/4} \exp\left(2 \sqrt{s_{\mathrm{tot}}(r)\cdot n} \right),
\end{multline*}
as $n\to \infty$.
\end{proof}

\subsection{The number of subgroups}

Theorem \ref{thm_permrep} also allows us to determine the asymptotic of the number of subgroups of a free product of virtually cyclic Coxeter groups. We start with the following lemma:
\begin{lem}\label{lem_fastgrowth}
  Suppose $\Gamma$ is a group so that
  \[h_n(\Gamma) \sim f(n) (n!)^\alpha\]
  as $n\to \infty$, where $\alpha >1$ and $f(n) \sim An^Ce^{B\sqrt n}$ as $n \to +\infty$. Then
  \[s_n(\Gamma) \sim n\cdot (n!)^{\alpha -1} f(n)\]
  as $n\to \infty$. 
\end{lem}

\begin{proof} The proof for the free group in \cite[Section 2.1]{Lubotzky_Segal} applies almost verbatim to our situation: following it we get first that
  \[
  h_n(\Gamma) - t_n(\Gamma) \le (n!)^\alpha \sum_{k=1}^{n-1} \binom n k^{1 - \alpha} f(k)f(n-k)
  \]
  and using their inequality $\binom n k \ge 2^{k-1} n/2$ and the asymptotic equivalent for $f$ we get that
  \[
  \frac{h_n(\Gamma) - t_n(\Gamma)}{(n!)^\alpha} \le \frac {4 + o(1)} n \sum_{k=1}^{n-1} 2^{-k} \left(\frac{k(n-k)} n \right)^C e^{-B(\sqrt n - \sqrt k - \sqrt{n-k})}. 
  \]
  It remains to see that the sum on the right is $o(n)$. This is easily done by separating in $k, n-k \ge n^{1/2+\varepsilon}$ (for some $\varepsilon > 0$ small enough), terms for which the summand is $\le 2^{(1-\delta)n}$ for some $0 < \delta < 1$ (depending on $\varepsilon$) and the remainder $O(n^{1/2+\varepsilon})$ terms which are $O(1)$ individually. 
\end{proof}
\begin{thm}\label{thm_subgrps} Let $m\in \NN_{\geq 2}$, $r_1,\ldots,r_m \in \NN$ and
\[\Gamma = \Asterisk_{l=1}^m \cox(\mathcal{A}_{r_l}) .\]
Then,
\[ s_n(\Gamma) \sim A_\Gamma \;n^{1+C_\Gamma}\; \exp(B_\Gamma\;\sqrt{n}) \; n!^{m-1}.\]
as $n\to\infty$.
\end{thm}

\begin{proof} This follows directly from Theorem \ref{thm_permrep} and Lemma \ref{lem_fastgrowth}.
\end{proof}

\subsection{Proof of Proposition \ref{prp_coeff}}\label{sec_proof1}

\begin{prprep}{\ref{prp_coeff}} For $r\in\NN_{\geq 1}$, let $b_{1,j},b_{2,j} >0$ and $k_j\in\NN$ for $j=0,\ldots ,r$, so that $k_0=1$ and 
\[k_0 < k_j \;\;\text{for all}\; j\in \{1,2,3,\ldots,r\}.\] 
Moreover, let
\[F : \DD_1 \to \CC\]
be defined by
\[F(z) = \sum_{n\geq 0} f_n \; z^n = \prod_{j=0}^r \left(1-z^{2k_j}\right)^{-b_{1,j}}\exp\left( \frac{b_{2,j}}{1-z^{k_j}} \right). \]
Then
\[f_n \rho^n =\frac{F(\rho)}{2\sqrt{\pi B(\rho)}}\exp\left(-\frac{1}{4}\frac{\left(A(\rho)-n\right)^2}{B(\rho)}\right)  \; \big(1+o(1)\big) \]
as $\rho\to 1$ uniformly for all natural numbers $n\in\NN$, where
\[
A(\rho) = \sum_{j=0}^r b_{1,j}\;2k_j \;\frac{\rho^{2k_j} }{1-\rho^{2k_j} } + b_{2,j} \;k_j\; \frac{\rho^{k_j}}{(1-\rho^{k_j})^2}\]
and
\[B(\rho) =  \sum_{j=0}^r 4\; b_{1,j}\; k_j^2\; \left(\frac{\rho^{2k_j} }{1-\rho^{2k_j} } + \left(\frac{\rho^{2k_j} }{1-\rho^{2k_j}}\right)^2 \right)    +\; b_{2,j} \;k_j^2\; \left( \frac{\rho^{k_j}}{(1-\rho^{k_j} )^2}  + 2\;\frac{\rho^{2k_j} }{(1- \rho^{k_j})^3}\right) .
\]
\end{prprep} 

\begin{proof} Our goal is of course to prove that $F$ is H-admissible. 

By assumption, all the singularities of $F$ lie on the unit circle $\abs{z} = 1$. 

It follows from the conditions (H1-3) that, for $\rho$ close enough to $1$, the maximum of an H-admissible function on the circle of radius $\rho$ is realized at $z=\rho$.  The assumption that $1=k_0 < k_j$ for all $j>0$ implies that this is indeed the case. 

Writing
\[P(z) = \sum_{j=0}^r -b_{1,j} \log(1-z^{2k_j}) +  \frac{b_{2,j}}{1-z^{k_j}}\]
we obtain $F(z) = \exp (P(z))$. As such, our first goal is to understand the behavior of
\[ P(\rho e^{i\theta}) - P(\rho)\]
as a function of $\theta$ near $\theta = 0$.

Let us consider $P$ term by term. We have
\[ \frac{1-\rho^{2k_j} e^{i \theta 2k_j}}{1-\rho^{2k_j}} = 1 + \frac{\rho^{2k_j}}{1-\rho^{2k_j}}(1 - e^{i \theta 2k_j}).  \]
Using a Taylor expansion, we obtain
\[1- e^{i \theta 2k_j} = -i\theta  2k_j + 2 \theta^2 k_j^2 + \bigO{\theta^3}\]
and hence
\[ \frac{1-\rho^{2k_j} e^{i \theta 2k_j}}{1-\rho^{2k_j} } = 1 - \frac{\rho^{2k_j}}{1-\rho^{2k_j} }\left(i \theta 2k_j - 2 \theta^2 k_j^2 + \bigO{\theta^3}\right).  \]
Again using a Taylor expansion, we have
\[ \log(1-z) = -z + -\frac{1}{2}z^2 + \bigO{z^3} \]
as $z\to 0$. So
\begin{eqnarray*}
\log\left(\frac{1-\rho^{2k_j} e^{i \theta 2k_j}}{1-\rho^{2k_j} }\right) & = &  -\frac{\rho^{2k_j} }{1-\rho^{2k_j}} \left(i\theta \; 2k_j - 2 \theta^2 k_j^2 \right) \\
& & + 2k_j^2 \left(\frac{\rho^{2k_j}}{1-\rho^{2k_j} }\right)^2 \theta^2  + \bigO{\left(\frac{\theta}{1-\rho^{2k_j}}\right)^3}
\\
& = & - i\;2k_j \;\frac{\rho^{2k_j} }{1-\rho^{2k_j} } \; \theta   + 2k_j^2 \left(\frac{\rho^{2k_j}}{1-\rho^{2k_j} } + \left(\frac{\rho^{2k_j} }{1-\rho^{2k_j}}\right)^2 \right)\;  \theta^2\\
& & + \bigO{\left(\frac{\theta}{1-\rho^{2k_j} }\right)^3}
\end{eqnarray*}
For the second type of terms in $P(z)$ we have
\[ \frac{1}{1-\rho^{k_j} e^{i\theta k_j}} - \frac{1}{1-\rho^{k_j} } =  \frac{1}{1-\rho^{k_j} }  \frac{ \rho^{k_j}  \left(e^{i\theta k_j}-1\right)}{1- \rho^{k_j} e^{i\theta k_j}}.\]
We have
\begin{eqnarray*}
\frac{1}{1-\rho^{k_j} e^{i\theta k_j}} & = & \frac{1}{1-\rho^{k_j} - \rho^{k_j}  (e^{i \theta k_j}-1)} \\
 & = & \frac{1}{1-\rho^{k_j} }  \frac{1}{1 - \frac{ \rho^{k_j}  (e^{i\theta k_j}-1)}{1- \rho^{k_j} }} .
\end{eqnarray*}
With yet another Taylor expansion, we get
\begin{eqnarray*}
\frac{1}{1-\rho^{k_j} e^{i\theta k_j}} & = &  \frac{1}{1-\rho^{k_j} } \left(1 + \frac{\rho^{k_j}  (e^{i \theta k_j}-1)}{1-\rho^{k_j} } + \left(\frac{\rho^{k_j}  (e^{i\theta k_j}-1)}{1-\rho^{k_j} }\right)^2 \right. \\
& & \left.+ \bigO{\left(\frac{e^{i\theta k_j}-1}{1- \rho^{k_j} }\right)^3} \right) .
\end{eqnarray*}
Using the expansion for $e^{i\theta k_j} - 1$ again, we obtain
\begin{eqnarray*}
\frac{1}{1- \rho^{k_j} e^{i\theta k_j}} & = &  \frac{1}{1- \rho^{k_j} } \left(1 + \frac{ \rho^{k_j} }{1- \rho^{k_j} } ( i \theta k_j - \frac{1}{2}\theta^2 k_j^2 ) \right. \\
& & \left.  - \frac{1}{2} \left(\frac{\rho^{k_j} }{1-\; \rho^{k_j} }\right)^2 k_j^2 \theta^2  + \bigO{\left(\frac{\theta }{1-\rho^{k_j} }\right)^3} \right) .
\end{eqnarray*}
Putting everything together yields
\begin{eqnarray*}
\frac{1}{1-\rho^{k_j} e^{i\theta k_j}} - \frac{1}{1-\rho^{k_j} } & = &  \frac{\rho^{k_j} \left(e^{i\theta k_j}-1\right)}{(1-\rho^{k_j}  )^2}  \left(1  + \frac{\rho^{k_j}}{1-\rho^{k_j} } ( i \theta k_j - \frac{1}{2}\theta^2 k_j^2 ) \right. \\
& & \left.   - \frac{1}{2} \left(\frac{\rho^{k_j} }{1- \rho^{k_j} }\right)^2 k_j^2 \theta^2 \right.  \left.  + \bigO{\left(\frac{\theta }{1-\rho^{k_j} }\right)^3} \right) \\
& = & \frac{\rho^{k_j} }{(1-\rho^{k_j} )^2}  (i \theta k_j - \frac{1}{2}\theta^2 k_j^2) -  \frac{\rho^{2k_j}}{(1-\rho^{k_j} )^3} \theta^2 k_j^2    \\
& &   + \bigO{\left(\frac{\theta}{1-\rho^{k_j} }\right)^3} \\ 
& = & i\frac{\rho^{k_j} }{(1-\rho^{k_j} )^2}   \theta k_j    -  \left(\frac{1}{2} \frac{\rho^{k_j} }{(1- \rho^{k_j})^2}  + \frac{\rho^{2k_j}}{(1-\rho^{k_j} )^3}\right) \theta^2 k_j^2 \\
& &  + \bigO{\left(\frac{\theta }{1-\rho^{k_j} }\right)^3}.
\end{eqnarray*}
As such, we have
\[
 P(\rho e^{i\theta}) - P(\rho)   =  i\; A(\rho)\; \theta  - \frac{1}{2} \; B(\rho)\; \theta^2  + \bigO{\left(\frac{\theta }{1-\rho^2}\right)^3} 
\]
where 
\[
A(\rho) = 
\sum_{j=0}^r b_{1,j} \;2k_j \;\frac{\rho^{2k_j} }{1-\rho^{2k_j} } + b_{2,j} \;k_j\; \frac{\rho^{k_j}}{(1- \rho^{k_j} )^2}
\]
and
\[
B(\rho) = 
\sum_{j=0}^r 4\; b_{1,j}\; k_j^2 \left(\frac{\rho^{2k_j} }{1-\rho^{2k_j} }  + \left(\frac{\rho^{2k_j} }{1-\rho^{2k_j}}\right)^2 \right)+ \; b_{2,j} \;k_j^2\; \left(\frac{\rho^{k_j} }{(1- \rho^{k_j} )^2}  + 2\;\frac{\rho^{2k_j} }{(1- \rho^{k_j})^3}\right) .
\]
Note that $A(\rho)$ and $B(\rho)$ are exactly of the form of \eqref{eq_functionsH}. This is of course no coincidence.

We now set
\[\delta(\rho) = (1-\rho)^{3/2} \; \left(\log\left(\frac{1}{1- \rho}\right)\right)^{3/4}.
\]
Note that this function is small enough for (H1) to hold. (H3) also readily follows from the form of $B(\rho)$.

All that remains is to check (H2). Indeed, we need to check that
\[\frac{\abs{F(\rho e^{i\theta})} \sqrt{B(\rho)}}{F(\rho)} \to 0 \]
uniformly in $\delta\leq \abs{\theta} \leq \pi$, we claim that this is guaranteed from our choice of $\delta(r)$. 

To this end, set $K=\mathrm{lcm}\st{k_j}{j=0,\ldots,r}$ and
\[\theta_j = j \cdot 2\pi/K\]
for $j=0,\ldots, K$. Once we prove that 
\[\frac{\abs{F(\rho e^{i\delta(\rho)})} \sqrt{B(\rho)}}{F(\rho)} \to 0 \]
and
\[\frac{\abs{F(\rho e^{i \theta_j} )} \sqrt{B(\rho)}}{F(\rho)} \to 0 \]
for all $0 < j <K$, we are done, since $F(z)$ has no other singularities on the circle $\abs{z}=1$.

For the first of these, we can use our approximation of $P(\rho e^{i\theta})$ for $\theta$ close to $0$. Indeed, since $\delta(\rho) / (1- \rho^2) \to 0$ as $\rho\to 1$ and $i\; A(\rho)\delta(\rho)$ is purely imaginary, we have that 
\[
\frac{\abs{F(\rho e^{i \delta(\rho)})} \sqrt{B(\rho)}}{F(\rho)}  \leq  \exp(-B(\rho)\delta(\rho)^2 + \frac{1}{2}\log(B(\rho))+C)\]
for some $C>0$ independent of $\rho$. It follows from the form of $\delta(\rho)$
that there exist constants $C' ,C''>0$ so that
\[-B(\rho)\delta(\rho)^2 + \frac{1}{2}\log(B(\rho)) \leq -C'\;\log\left(\frac{1}{1- \rho}\right)^{3/2}+C''\log\left(\frac{1}{1- \rho}\right) \to -\infty\]
as $\rho \to 1$.

For the second type of terms we have
\[\frac{\abs{F(\rho e^{i \theta_j} )} \sqrt{B(\rho)}}{F(\rho)} = \abs{\exp\left(\sum_{\substack{0\leq i \leq r \\ j\cdot k_i \nmid K}} \frac{b_{2,i}}{1-\rho^{k_i} e^{j\cdot k_i \cdot 2\pi/K}} - \frac{b_{2,i}}{1-\rho^{k_i}} + b_{1,i} \log\left(\frac{1-\rho^{k_i}}{1-\rho^{k_i} e^{j\cdot k_i \cdot 2\pi/K}} \right) \right)}. \]
Whenever $K$ does not divide $j\cdot k_i$, $1/(1-\rho^{k_i} e^{j\cdot k_i \cdot 2\pi/K})$ and $\log(1-\rho^{k_i} e^{j\cdot k_i \cdot 2\pi/K})$ are bounded as functions of $\rho$. So there exists a constant $C>0$ such that:
\[\frac{\abs{F(\rho e^{i \theta_j} )} \sqrt{B(\rho)}}{F(\rho)} \leq \abs{\exp\left(\sum_{\substack{0\leq i \leq r \\ j\cdot k_i \nmid K}} - \frac{b_{2,i}}{1-\rho^{k_i}} + b_{1,i} \log\left(1-\rho^{k_i} \right) + C \right)} \to 0, \]
as $\rho\to 1$, which proves that $F$ is indeed H-admissible.
\end{proof}

\subsection{Proof of Lemma \ref{lem_roots}}\label{sec_proof2}

\begin{lemrep}{\ref{lem_roots}}
Let $r\in\NN$, $\alpha_j, \beta_j \in (0,\infty)$ and $k_j\in \NN_{\geq 1}$, for $j=0,\ldots,r$. Consider for $t\in (0,1)$ the equation in $y$ given by
\[t\;\sum_{j=0}^r \alpha_j \; \frac{y^{2k_j}}{1-y^{2k_j}} + \beta_j\; \frac{y^{k_j}}{(1-y^{k_j})^2} = 1.\]
This equation has a solution $y=y(t)$ that satisfies
\[y (t) = 1 - c_1 t^{1/2} + c_2 t +\bigO{t^\gamma} \]
as $t\to 0$ for some $\gamma \in \QQ_{>1}$, where the constant $c_1>0$ is given by
\[c_1 = \sqrt{\sum_{j=0}^r\frac{\beta_j}{k_j^2}}.\]
\end{lemrep}

\begin{proof} We will turn the equation into a polynomial equation (depending on the parameter $t$) and then develop the Puiseux series for $y(t)$. Recall that Puiseux's theorem tells us that we can find an $m\in \NN$ and $k_0\in\ZZ$ so that
\[y(t) = \sum_{k\geq k_0} c_k\; t^{k/m}.\]
We will apply Newton's method to find the first three coefficients near $1$. 

This method can be given a nice geometric description. Namely, given a polynomial equation of the form
\begin{equation}\label{eq_genpuis}
F(t,y) = \sum_{k=0}^r A_k(t)\;y^k = 0,
\end{equation}
the first power $\gamma_0$ of $t$ in $y$ can be found by considering the Newton polytope $\mathcal{P}\subset \mathbb{R}^2$ of this polynomial. This polytope is the convex hull of all points $(\mathrm{ord}(A_k(t)),k)$, where $\mathrm{ord}(A_k(t))$ denotes the lowest power of $t$ that appears in $A_k(t)$. Now consider all the sides of $\mathcal{P}$. If the line $L$ spanned by a side of $\mathcal{P}$ supports $\mathcal{P}$ and $\mathcal{P}$ lies above $L$, then $-\gamma_1$, where $\gamma_1$ is the slope of $L$, is the first power of a branch of solutions to the equation.

Writing $y=c_1\; t^{\gamma_1}+y_1$, we obtain a new polynomial equation
\[F(t,c_1\; t^{\gamma_1}+y_1)=0\]
in $y_1$ and the process can be iterated, with the only condition that at each step, only those lines are considered that have a larger negative slope than the slopes that have already appeared. Since the powers of $t$ increase in each iteration, this allows us to compute $y(t)$ up to any order. For more information on these methods see \cite{Basu_Pollack_Roy_book}.

Now we return to our equation, which is equivalent to
\[t \sum_{j=0}^r\Big(\alpha_j (1-y^{k_j}) y^{2k_j} + \beta_j (1+y^{k_j})y^{k_j} \Big)\prod_{l\neq j}(1-y^{k_l})^2(1+y^{k_l}) = \prod_{j=0}^r (1-y^{k_j})^2 (1+y^{k_j}).\]
Writing $y=1+y_1$, we obtain
\begin{multline*}
t \sum_{j=0}^r \Bigg\{\Big(\alpha_j (1-(1+y_1)^{k_j}) (1+y_1)^{2k_j} + \beta_j (1+(1+y_1)^{k_j})(1+y_1)^{k_j} \Big) \\
 \cdot \prod_{l\neq j}(1-(1+y_1)^{k_l})^2(1+(1+y_1)^{k_l}) \Bigg\} = \prod_{j=0}^r (1-(1+y_1)^{k_j})^2 (1+(1+y_1)^{k_j}).
\end{multline*}
The lowest power of $y_1$ on the right hand side of the equation is $y_1^{2r}$. On the left hand side, this is $y_1^{2r+2}$. This means that $(2r,1)$, $(2r+1,1)$ and $(2r+2,0)$ are the first vertices of the Newton polytope and hence that $\gamma_1 = 1/2$. So $y=1+c_1 t^{1/2} + \bigO{t^{\gamma'}}$ for some $\gamma'> \frac{1}{2}$. 
To obtain the coefficient $c_1$, we now solve equate the terms that lie on the line of slope $-\gamma_1$ supporting the Newton polygon. In the notation of \eqref{eq_genpuis}, we need to solve
\[ \sum_{\mathrm{ord}(A_k) + \gamma \cdot k = r + 1} a_k c_1^k =0,\]
Here the $r+1$ is the lowest power of $t$ that appears and $a_k$ is so that 
\[A_k(t) = a_k\;t^{\mathrm{ord}(A_k)} + \text{higher order terms}.\]
So, for our polynomial, we get the equation
\[
 \sum_{j=0}^r \beta_j \cdot 2 \prod_{l\neq j} 2\;k_l^2 c_1^2 = \sum_{j=0}^r \frac{\beta_j}{c_1^2 k_j^2}  \prod_{l=0}^r 2\;k_l^2 c_1^2   = \prod_{j=0}^r 2\;k_j^2 c_1^2 
\]
and hence
\[c_1^2 = \sum_{j=0}^r \frac{\beta_j}{k_j^2}.\]
We choose the branch corresponding to the negative solution of this equation.

Writing  $y=1-c_1\;t^{1/2}+y_2$ for the second iteration, we obtain
\begin{multline*}
t \sum_{j=0}^r \Bigg\{\bigg(\alpha_j (1-(1-c_1\;t^{1/2}+y_2)^{k_j}) (1-c_1\;t^{1/2}+y_2)^{2k_j} \\
+ \beta_j (1+(1-c_1\;t^{1/2}+y_2)^{k_j})(1-c_1\;t^{1/2}+y_2)^{k_j} \bigg) \\
\cdot \prod_{l\neq j}(1-(1-c_1\;t^{1/2}+y_2)^{k_l})^2(1+(1-c_1\;t^{1/2}+y_2)^{k_l}) \Bigg\} \\
 = \prod_{j=0}^r (1-(1-c_1\;t^{1/2}+y_2)^{k_j})^2 (1+(1-c_1\;t^{1/2}+y_2)^{k_j}).
\end{multline*}

Since the power $t^{r+1}$ in the $y_2^0$ coefficient disappears, the lowest power in the constant coefficient is $t^{r+3/2}$. For $0<m \leq 2r+2$, the lowest power of $t$ in the $y_2^m$-term is $t^{r+1-m/2}$. So the Newton polytope contains the vertices $(0,r+3/2)$, $(m,r+1-m/2)$ for $m=1,\ldots, 2r+2$. This implies that $\gamma_2 = 1$.

So we may write $y_2 = c_2 t + \bigO{t^\gamma}$ for some $\gamma>1$. We could determine the constant $c_2$ in a similar fashion to how we determined $c_1$. It however turns out that in our application, we will not need the value of the constant.
\end{proof}

\subsection{Proof of Theorem \ref{thm_asymp}}\label{sec_proof3}

\begin{thmrep}{\ref{thm_asymp}} Let $(f_n)_{n\in \NN}$ be as above. Then
\[
f_n   \sim  C_1(b,k,r)\cdot n^{-3/4 + \sum_{j=0}^r b_{1,j}/2} \exp\left(C_2(b,k,r)\cdot\sqrt{n}\right) ,
\]
as $n\to\infty$, where
\[
C_1(b,k,r) =  \frac{ \left(\sum_{j=0}^r \frac{b_{2,j}}{k_j}\right)^{1/4} }{\sqrt{4\pi}}\exp\left(\sum_{j=0}^r \frac{b_{2,j}}{2k_j} + \sum_{j=0}^r \frac{k_j-1}{2k_j}\right) \prod_{j=0}^r \left(\frac{1}{2k_j\; \sqrt{\sum_{j=0}^r \frac{b_{2,j}}{k_j}}} \right)^{b_{1,j}} 
\]
and
\[C_2(b,k,r) = 2\sqrt{\sum_{j=0}^r \frac{b_{2,j}}{k_j}}.\]
\end{thmrep}

\begin{proof} The standard trick is to find a sequence $(\rho_n)_n$ so that
\[A(\rho_n) = n. \]
We claim that there exists a choice of $(\rho_n)_n$ so that $\rho_n\in (0,1)$ for $n$ large enough and
\[ \rho_n \to 1,\]
as $n\to \infty$. As such, we can then apply Proposition \ref{prp_coeff} and we obtain that
\[f_n \sim \frac{F(\rho_n)}{\sqrt{2\pi B(\rho_n)}\; \rho_n^n} \]
as $n\to \infty$.

Let us start by finding $\rho_n$. Using Proposition \ref{prp_coeff}, we obtain
\[
\sum_{j=0}^r b_{1,j}\;2k_j \;\frac{\rho_n^{2k_j} }{1-\rho_n^{2k_j} } + b_{2,j} \;k_j\; \frac{\rho_n^{k_j}}{(1-\rho_n^{k_j} )^2} = n\]
Write $t=1/n$.

So we obtain from Lemma \ref{lem_roots} that
\[\rho_n = 1 - \frac{c_1}{\sqrt{n}} + \frac{c_2}{n} + \bigO{n^{\gamma}}, \]
as $n\to \infty$ for some $\gamma>1$, where $c_1$ and $c_2$ are the constants in that lemma, with $\alpha_j =  b_{1,j}\;2k_j$ and $\beta_j = b_{2,j}\; k_j$. So, in particular
\[c_1 = \sqrt{\sum_{j=0}^r \frac{b_{2,j}}{k_j}}.\]
So we obtain
\[\rho_n^n \sim \left( 1 - \frac{c_1}{\sqrt{n}} + \frac{c_2}{n} \right)^n \sim \exp\left(-\frac{c_1^2}{2}+ c_2 - c_1 \sqrt{n}\right),\]
as $n\to\infty$. Because $\rho_n\to 1$ as $n\to \infty$,
\[
B(\rho_n) \sim  \sum_{j=0}^r 2\; b_{2,j}\; k_j^2 \frac{1}{(1-\rho_n^{k_j})^3} 
  \sim   \frac{n^{3/2}}{c_1^3}  \sum_{j=0}^r \frac{2\; b_{2,j}}{k_j} 
\]
as $n\to \infty$. Likewise,
\begin{eqnarray*}
F(\rho_n) & = & \prod_{j=0}^r \left(1 - \rho_n^{2k_j}\right)^{-b_{1,j}} \exp\left(\frac{b_{2,j}}{1- \rho_n^{k_j}} \right)\\
 & \sim &  \prod_{j=0}^r \left(\frac{n^{1/2}}{2k_j\; c_1}\right)^{b_{1,j}} \exp\left(\frac{b_{2,j}}{k_j\; c_1\;n^{-1/2} - k_j\; c_2\;n^{-1} - \binom{k_j}{2}\;c_1^2\; n^{-1}} \right)\\
 & \sim &  \prod_{j=0}^r \left(\frac{n^{1/2}}{2k_j\; c_1}\right)^{b_{1,j}} \exp\left(\frac{b_{2,j}}{k_j\; c_1}\;n^{1/2} + \frac{b_{2,j}\; c_2}{k_j\; c_1^2}  + \frac{k_j-1}{2k_j} \right)
\end{eqnarray*}
as $n\to\infty$.
Putting all of the above together, we obtain
\begin{multline*}
f_n   \sim  \left(\frac{4\pi}{c_1^3}  \sum_{j=0}^r  \frac{b_{2,j}}{k_j}\right)^{-1/2} \exp\left(\frac{c_1^2}{2}- c_2 + \sum_{j=0}^r\frac{b_{2,j}\;c_2}{k_j\;c_1^2} + \frac{k_j-1}{2k_j} \right) \prod_{j=0}^r \left(\frac{1}{2k_j\; c_1} \right)^{b_{1,j}} \\
\cdot n^{-3/4 + \sum_{j=0}^r b_{1,j}/2}\exp\left(\sqrt{n}\left(c_1 + \frac{1}{c_1} \sum_{j=0}^r \frac{b_{2,j}}{k_j}\right)\right) \\
\end{multline*}
Now we see why we don't need the value of $c_2$: it cancels, and we obtain that
\begin{multline*}
f_n   \sim \frac{1}{\sqrt{4\pi}} \left(\sum_{j=0}^r \frac{b_{2,j}}{k_j}\right)^{1/4} \exp\left(\sum_{j=0}^r \frac{b_{2,j}}{2k_j} + \sum_{j=0}^r \frac{k_j-1}{2k_j}\right) \prod_{j=0}^r \left(\frac{1}{2k_j\; c_1} \right)^{b_{1,j}} \\
\cdot  n^{-3/4 + \sum_{j=0}^r b_{1,j}/2}\exp\left(2 \sqrt{\sum_{j=0}^r \frac{b_{2,j}}{k_j}}\;\sqrt{n}\right)
\end{multline*}
as $n\to \infty$, which proves our claim.
\end{proof}

\bibliography{bib}{}
\bibliographystyle{amsplain}
\end{document}

%% file: table_values.tex
\begin{table}[H]
\begin{tabular}{c|c c c |c c}
 &  \multicolumn{3}{c}{Exact values} &  \multicolumn{2}{c}{Numerical values}  \\
\hline 
\rule{0pt}{4ex}
$r$ & $A_r$ & $B_r$ & $C_r$ & $A_r$ & $B_r$ \\
\hline 
\hline
\rule{0pt}{4ex}
0 & $\frac{1}{\sqrt{8\;\pi}}\cdot \exp\left(-\frac{1}{2}\right)$ & $2$ & $-\frac{1}{2}$ & $0.1210\ldots$ & $2$ \\[3mm]
\hline 
\rule{0pt}{4ex}
1 & $ \frac{1}{\sqrt[4]{2048\;\pi^2}}\cdot \exp\left(-\frac{7}{4}\right)$ & $2\sqrt{2}$ & $-\frac{1}{4}$ &  $0.01457\ldots$ & $2.8284\ldots$ \\[3mm]
\hline 
\rule{0pt}{4ex}
2 & $\frac{1}{320\;\sqrt{\pi}}\cdot \exp\left(-\frac{63}{8}\right)$ & $2\sqrt{5}$ & $\frac{1}{2}$ & $6.7020 \cdot 10^{-7}$ & $4.4721\ldots$  \\[3mm]
\hline 
\rule{0pt} {4ex}
3 &  $\frac{1}{68719476736\sqrt{\pi}}\exp\left(-\frac{671}{16}\right)$ & $8$ & $\frac{13}{4}$ & $5.0248 \cdot 10^{-30}$ & $8$ \\[3mm]
\hline
\end{tabular}
\caption{The first four values of $A_r$, $B_r$ and $C_r$.}
\end{table}